\documentclass[sn-mathphys-num]{sn-jnl}

\usepackage{
  amsmath, amsthm, amssymb, mathtools, dsfont, units,          
  graphicx, wrapfig, subfig, float,                            
  listings, color, inconsolata, pythonhighlight,               
  fancyhdr, hyperref, enumerate, enumitem, framed, xcolor}   
  
\usepackage{mathrsfs}

\usepackage{pgfplots}

\usepackage{comment}
\usepackage{tikz-cd}
\tikzcdset{scale cd/.style={every label/.append style={scale=#1},
    cells={nodes={scale=#1}}}}

\usepackage[bottom]{footmisc}

\usepackage{filecontents}

\allowdisplaybreaks

\usepackage{geometry}

\hypersetup{colorlinks=true, linkcolor=magenta}

\renewcommand{\refname}{\normalsize\bf{References}}

\definecolor{greenText}{rgb}{0.5, 0.7, 0.5}
\definecolor{greyText}{rgb}{0.5, 0.5, 0.5}
\definecolor{codeFrame}{rgb}{0.7, 0.5, 0.5}
\definecolor{backgroundColor}{rgb}{0.95, 0.95, 0.92}
\lstdefinestyle{code} {
  frame=single,
  rulecolor=\color{greenText},
  numberstyle=\tiny\color{greyText},
  commentstyle=\color{greenText},
  basicstyle=\linespread{1}\ttfamily\footnotesize,
  keywordstyle=\ttfamily\footnotesize,
  showstringspaces=false,
  numbers=left,
  numbersep=8pt,
  xleftmargin=1.95em,
  framexleftmargin=1.6em }
\let\mc\mathcal

\newcommand{\ul}{\underline}                            

\renewcommand{\l}{\left}                                
\renewcommand{\r}{\right}                               

\newcommand{\la}{\langle}                               
\newcommand{\ra}{\rangle}                               
\newcommand{\p}{\prime}                                 
\newcommand{\tl}{\triangleleft}                         
\newcommand{\tr}{\triangleright}                        

\renewcommand{\a}{\alpha}                               
\renewcommand{\b}{\beta}                                
\renewcommand{\d}{\delta}                               
\renewcommand{\epsilon}{\varepsilon}                    
\newcommand{\e}{\epsilon}                               
\renewcommand{\phi}{\varphi}                            

\newcommand{\Id}{\text{Id}}                             

\newcommand{\Inn}{\text{Inn}}                           
\newcommand{\Aut}{\text{Aut}}                           

\newcommand{\Conj}{\text{Conj}}                         
\newcommand{\Adconj}{\text{Adconj}}                     
\newcommand{\Tak}{\text{Tak}}                           

\newcommand{\pInn}{\widehat{\text{Inn}}}

\theoremstyle{definition}
\newtheorem{theorem}{Theorem}[section]
\newtheorem{proposition}[theorem]{Proposition}
\newtheorem{lemma}[theorem]{Lemma}

\newtheorem{definition}[theorem]{Definition}
\newtheorem{example}[theorem]{Example}
\newtheorem{counterexample}[theorem]{Counterexample}
\newtheorem{remark}[theorem]{Remark}

\pgfplotsset{compat=1.18}

\begin{document}

\title[Profinite Quandles]{On Profinite Quandles}

\author[1]{Alexander W. Byard}
\affil[1]{\orgdiv{Department of Mathematics},\orgname{ University of Chicago}, \street{Eckhart Hall,
5734 S University Ave}, \city{Chicago}, \state{IL} \postcode{60637}, \country{USA}}
\email{abyard@uchicago.edu}

\author[2]{Brian Cai}
\affil[2]{\orgdiv{Department of Mathematics},\orgname{University of Florida},
\street{P.O.\ Box 118105},
\city{Gainesville}, \state{FL} \postcode{32611} \country{USA}}
\email{brianfl2017@gmail.com}

\author[3]{Nathan P. Jones}
\affil[3]{\orgdiv{Department of Mathematics},\orgname{ University of Kentucky},\street{ Patterson Office Tower, 120 Patterson Dr.},\city{ Lexington},\state{ KY} \postcode{ 40508},\country{ USA}}
\email{Nathan.Jonesviolin@uky.edu}

\author[4]{Lucy H. Vuong}
\affil[4]{\orgdiv{Department of Mathematics},\orgname{Harvard University},
\street{Science Center Room 325,
1 Oxford Street},
\city{Cambridge}, \state{MA} \postcode{02138} \country{USA}}
\email{lucyvuong@college.harvard.edu}

\author*[5]{David N. Yetter}
\affil*[5]{\orgdiv{Department of Mathematics},
\orgname{Kansas State University},
\street{1228 N. Martin Luther King, Jr. Drive},
\city{Manhattan}, \state{KS} \postcode{66506} \country{USA}}
\email{dyetter@ksu.edu}



\maketitle

\begin{abstract}
    s We undertake the study of profinite quandles. A quandle $(Q, \triangleleft)$ is a set $Q$ with a operation $\triangleleft$ that satisfies idempotency, right invertibility, and right self-distributivity axioms. Profinite quandles are inverse limits of inverse systems of finite quandles and are closely related to profinite groups. As with all profinite models of algebraic theories, they are naturally equipped with the structure of a Stone space, with respect to which structure their operations are continuous. We provide several constructions of profinite quandles from profinite groups, and from other profinite quandle.  We characterize which subquandles of profinite quandles are again profinite. Finally, we provide a characterization of algebraically connected profinite quandles in terms of the profinite completion of their inner automorphism groups $\widehat{\Inn(Q)}$. It is anticipated that the results herein will find applications to the \'{e}tale homotopy theory of number fields.
\end{abstract}


\section{Introduction}

Quandles were introduced by Joyce in his 1982 seminal paper \textit{A classifying invariant of knots, the knot quandle} as an algebraic invariant on classical knots and links.~\cite{JOYCE198237} However, the earliest known example of a quandle dates back to 1940s Japan when Mituhisa Takasaki defined a kei, now known as an involutory quandle, in an attempt to find a non-associative algebraic structure capable of describing reflections in finite geometry.~\cite{takasaki1943abstractions} Since then, quandles have been extensively studied in relation to classical knots and links~\cite{Carter2003quandle, book, Carter2004SurfacesI4, nelson2014link, article, Inoue2001QUANDLEHO} as well as other mathematical structures. For instance Zablow \cite{Zablow2003} showed that the Dehn twists of any orientable surface form a quandle, while Crans \cite{crans2004lie} showed that quandles arise naturally in relating a Lie group and its Lie algebra.  Yetter \cite{yetter2003} showed that many monodromy phenomena arising in geometric topology and algebraic geometry are more naturally described by quandle homomorphisms rather than group homomorphisms with unalgebraic side-conditions as had been done classically. For excellent introductory articles on quandles, see~\cite{carter2012survey}.

A quandle is a set $Q$ with two binary operations satisfying idempotency, right-invertibility, and right-distributivity axioms analogous to the Reidemeister moves in knot theory. A subquandle of $Q$ is a subset closed under the operations of $Q$, and the collection of subquandles of $Q$ forms a lattice. In their 2018 paper, Saki and Kiani demonstrated that the lattice of subquandles of a finite quandle is complemented and gave an example of an infinite quandle that does not have a complemented subquandle lattice.~\cite{saki2021complemented} Amsberry et al.  constructed an ind-finite quandle without a complemented subquandle lattice, but conjectured that the subquandle lattice of a profinite will be complemented.~\cite{amsberry2023complementation}

Our original motivation in studying profinite quandles was to resolve a conjecture of Amsberry et al.~\cite{amsberry2023complementation} which remains unresolved. We became aware, however, that the general theory of profinite quandles is of some importance as the \'{e}tale homotopy of number fields naturally gives rise to profinite quandles, which have been studied by Davis and Schlank \cite{Schlank2023}. 
In this paper we introduce the theory of profinite quandles. We show that products and finite disjoint unions preserve profiniteness and give a topological characterization for when a subquandle of a profinite quandle is profinite. Finally, we provide a version of algebraic connectedness appropriate to profinite quandles and characterize the structure of algebraically connected profinite quandles in terms of the profinite completion of their inner automorphism groups.  

While this paper was in its final preparation, we became aware that Singh \cite{S23} had independently derived a number of basic results about profinite quandles, mostly distinct from our own.  We note in footnotes those of our results which are also found in \cite{S23}.

\section{Basic Quandle Theory}

\begin{definition}[Quandle]
A quandle $(Q, \triangleleft, \triangleleft^{-1})$ is a set $Q$ equipped with two binary operations $\triangleleft,\triangleleft^{-1}:Q\times Q\to Q$ satisfying the following axioms for all $x,y,z \in Q$:
\begin{enumerate}[label=\textbf{(Q\arabic*)}]
\item $x \triangleleft x = x$,
\item $(x \triangleleft y) \triangleleft^{-1} y = x = (x \triangleleft^{-1} y) \triangleleft y$,
\item $(x \triangleleft y) \triangleleft z = (x \triangleleft z) \triangleleft (y \triangleleft z)$.
\end{enumerate}
In other words $\triangleleft$ is idempotent, right-invertible, and right-distributive.
\end{definition}

Notice that this definition describes quandles as models of an equational theory. Quandles can be equivalently defined using a first order theory with a single operation $\triangleleft$ since (\textbf{Q2}) is equivalent to the condition that for all $y\in Q$ the map
$x \mapsto x\triangleleft y$ is a bijection. This allows for the omission of $\triangleleft^{-1}$ when convenient.

An algebraic structure $(Q,\triangleleft)$ satisfying (\textbf{Q2}) and (\textbf{Q3}) is called a rack.

\begin{example}[Tait quandle]
Let $Q = \{1,2,3\}$ and define $\triangleleft : Q \times Q \to Q$ by the following operation table:

    \begin{center}
    \begin{tabular}{c | c c c }
        $\triangleleft$ & 1 & 2 & 3  \\
        \hline
        1 & 1 & 3 & 2  \\
        2 & 3 & 2 & 1  \\
        3 & 2 & 1 & 3  
    \end{tabular}
    \end{center}

With this operation $(Q,\triangleleft)$ forms the Tait Quandle.
\end{example}

\begin{example}
    (Trivial quandle). Let $S$ be a set and define $\triangleleft:S \times S \to S$ by $(x,y) \mapsto x$ for all $x,y \in S$.  Then $\triangleleft$ defines the trivial quandle structure on a set $S$.
\end{example}

\begin{example}[Conjugation quandle]
Let $G$ be a group. We define the operation $\triangleleft : G \times G \to G$ be given by $(g,h)\mapsto h^{-1}gh$ for all $g,h\in G$. The group $G$ with this operation defines a quandle $(G, \triangleleft)$, denoted by $\Conj(G)$, called the conjugation quandle on $G$.
\end{example}

Conj defines a functor from the category $\textbf{Gp}$ of groups to the category $\textbf{Qndl}$ of quandles. Furthermore, Conj admits a left adjoint, denoted Adconj. For all quandles $Q$ we have that $\Adconj(Q):=F(Q)/\sim$, where $F(Q)$ is the free group generated by $Q$ and $x \triangleleft y \sim y^{-1}xy$. \cite{JOYCE198237} In some literature, $\Adconj(Q)$ is also referred to as the universal augmentation group of $Q$.

\begin{definition}[Involutory quandle]
An involutory quandle or kei is a quandle $(Q,\triangleleft)$ satisfying $(x \triangleleft y) \triangleleft y = x$ for all $x,y \in Q$.
\end{definition}

\begin{example}[Takasaki quandle]
The Takasaki quandle is named after Takasaki's early work on kei.~\cite{takasaki1943abstractions} Let $A$ be an abelian group and define the operation $\triangleleft:A\times A\to A$ by $(x,y)\mapsto2y-x$. Together $(A,\triangleleft)$ defines a kei, denoted by $\Tak(A)$, called the Takasaki kei on $A$, or the Takasaki quandle on $A$.
\end{example}

Tak defines a functor from the category $\textbf{Ab}$ of abelian groups to the category $\textbf{Kei}$ of kei. Moreover, Tak admits a left adjoint, denoted AdTak. For all involutory quandles $Q$, we have that $\text{AdTak}(Q):=\la Q\ra_{\text{ab}}/\sim$, where $\la Q\ra_{\text{ab}}$ is the free abelian group generated by $Q$ and $x \triangleleft y \sim 2y - x$.~\cite{amsberry2023complementation}

\begin{example}[Core quandle]
The Takasaki quandle on abelian groups can be generalized to all groups.  For a group $G$, define the operation $\triangleleft:G\times G\to G$ by $(x,y)\mapsto xy^{-1}x$.  The pair $(G,\triangleleft)$ defines a kei structure on $G$ denoted Core($G$).  
\end{example}

\begin{definition}[Subquandle]
For a quandle $(Q, \triangleleft)$, a subset $Q' \subseteq Q$ is a subquandle of Q provided $(Q', \triangleleft)$ is a quandle, i.e. $Q^\p$ is closed under $\triangleleft|_{Q^\p}$ and $\triangleleft^{-1}|_{Q^\p}$. We write $Q^\prime\preceq Q$ to denote that $Q^\prime$ is a subquandle of $Q$. When $Q^\p$ is a proper subquandle of $Q$, we write $Q^\prime\prec Q$.
\end{definition}

\begin{proposition}
Let $Q$ be a quandle. If $Q_1$ and $Q_2$ are subquandles of $Q$, then $Q_1\cap Q_2$ is a subquandle of $Q$.
\end{proposition}

\begin{definition}
Let $(Q,\triangleleft)$ be a quandle. Given a subset $S \subseteq Q$, the subquandle $\la S\ra$ generated by $S$ is given by
\[\la S \ra := \bigcap_{S \subseteq Q^\prime \preceq Q} Q^\prime\]
\end{definition}

\begin{definition}[Quandle homomorphism]
Let $(Q_1,\triangleleft_1)$ and $(Q_2,\triangleleft_2)$ be quandles. A quandle homomorphism $\phi$ from $Q_1$ to $Q_2$ is a map $\varphi:Q_1\to Q_2$ such that $\varphi(x \triangleleft_1 y)=\varphi(x)\triangleleft_2\varphi(y)$. A quandle isomorphism is a bijective quandle homomorphism.
\end{definition}

\begin{definition}[Automorphism group of a quandle]
Let $Q$ be a quandle. A quandle isomorphism $\varphi:Q\to Q$ between $Q$ and itself is called a quandle automorphism of $Q$. The set of quandle automorphisms
\[\Aut(Q) := \{\varphi : Q \to Q : \varphi \text{ is an automorphism of $Q$}\}\]
of $Q$ forms a group called the automorphism group of $Q$.

Let $y\in Q$. The symmetry $S_y:Q\to Q$ at $y$ is given by $S_y(x)=x\tl y$ for all $x\in Q$. It is a quandle automorphism by axiom (\textbf{Q3}). The group
\[\Inn(Q):=\la S_y\ra_{y\in Q}\]
generated by all symmetries of $Q$ is called the inner automorphism group of $Q$. It is a theorem of Joyce \cite{JOYCE198237} that $\Inn(Q)$ is a normal subgroup of $\Aut(Q)$.
\end{definition}

Finally, we recall 

\begin{definition}
A quandle $Q$ is said to be \textit{algebraically connected} if there is exactly one orbit under right action by $\Inn(Q)$.
\end{definition}

In the context of finite quandles, algebraically connected quandles form ``building blocks'' from which all finite quandles can be iteratively assembled using semi-disjoint unions, that is quandle structures on the disjoint union, the operations of which restrict to the given ones on the the summands (see \cite{ehrman2008toward} for a formal definition).


In \cite{ehrman2008toward}, Ehrman et al. established a characterization of algebraically connected finite quandles. First, a proposition following from properties of group actions.

\begin{proposition}
Let $Q$ be a connected quandle on $n$ elements. Then $n$ divides $|\Inn(Q)|$ and any choice of $q \in Q$ induces a $\Inn(Q)$-equivariant bijection between $H\backslash\Inn(Q),$ where $H$ is the stabilizer of $q$ under action by $\Inn(Q)$.
\end{proposition}

From the proposition, we may represent $Q$ by $(\{Hg_1,...,Hg_n\},\triangleleft)$ where $Hg_i \triangleleft Hg_j = Hg_k$ if $q_i \triangleleft q_j = q_k$. Define an augmentation map $|\cdot| : Q = H\backslash\Inn(Q) \to \Inn(Q)$ by $|Hg_i| = g$ where $g \in \Inn(Q)$ such that $x \cdot g = x \triangleleft Hg_i$ for all $x \in Q$. To distinguish between the augmentation map and the order of a group, we will denote $|H|$ as the order of the subgroup $H$ and $|Hh|$ as the augmentation of $H$ as a right coset in $H \backslash \Inn(Q)$.

\begin{theorem}{\cite{ehrman2008toward}}
Let $Q$ be an algebraically connected quandle on $n$ elements, let $H$ be the stabilizer of $q \in Q$ under action by $\Inn(Q)$, and let $\{g_2,...,g_n\}$ be all coset representatives of $H$ not in $H$. Then $H \subset \Inn(Q) \subset \mathfrak{S}_n$, $\frac{|G|}{|H|} = n$, $|Hh| \in Z(H)$, and $\Inn(Q)$ is generated by $\{|Hh|,|Hg_2|,...,|Hg_n|\}$, where $|Hg_i| = g_i^{-1}|Hh|g_i$ for $2 \leq i \leq n$.
\end{theorem}

\begin{theorem}{\cite{ehrman2008toward}}
Suppose that for groups $G$ and $H$, we have $H \subset G \subset \mathfrak{S}_n$ and $\frac{|G|}{|H|} = n$ for $n \in \mathbb{N}$. Let $g_1,...,g_n$ be coset representatives of $H$ in $G$ and further suppose that $G$ is generated by $\{g_1^{-1}hg_1,...,g_n^{-1}hg_n\}$ for some $h \in Z(H)$. Then $(H \backslash G, \triangleleft)$ where $Hg_i \triangleleft Hg_j := Hg_ig_j^{-1}hg_j$ defines an algebraically connected quandle on $n$ elements.
\end{theorem}

Later, we will obtain analogous results for profinite quandles.

As quandles are models of an equational theory, the category of quandles admits all limits and colimits.  In particular, products:

\begin{definition}[Product]
Let $(Q, \tl_q)$ and $(S, \tl_S)$ be quandles. The product quandle $(Q \times S, \tr)$ is the product of the underlying sets with the operation  
\[(q_1, s_1) \tr (q_2, s_2) := (q_1 \tr_Q q_2, s_1 \tr_S s_2).\]

Let $(Q_i, \tl_i)$ for $i\in {\mathcal I}$ be a set of quandles, the product quandle $\prod_{i\in {\mathcal I}} Q_i$ is the product of the underlying sets with 
\[ \langle q_i \rangle_{i\in {\mathcal I}} \tl
\langle r_i \rangle_{i\in {\mathcal I}} := \langle q_i \tl_i r_i\rangle_{i\in {\mathcal I}}
\]
\end{definition}

The coproduct of quandles in complicated in the same way that the coproduct of groups (i.e. free product of groups) is, and as it is ill-behaved in the profinite setting, we will not consider it.

We will, however, consider a special case of the semi-disjoint union construction, to wit:

\begin{definition}[Disjoint Union]
Let $(Q, \tl_q)$ and $(S, \tl_S)$ be quandles. The disjoint union quandle $(Q\sqcup S, \tr)$ is defined as follows:
\begin{enumerate}[label=(\roman*)]
    \item If $q\in Q$ and $s\in S$, then $q\tl s=q$ and $s\tl q=s$.
    \item If $q_1,q_2\in Q$, then $q_1\tl q_2=q_1\tl_Qq_2$.
    \item If $s_1,s_2\in S$, then $s_1\tl s_2=s_1\tl_Ss_2$.
\end{enumerate}
\end{definition}

It should be noted that disjoint union is {\em not} the coproduct in the category of quandles. It can be described in terms of a univeral property, being the quandle equipped with quandle homomorphisms from $Q$ and $S$ and universal among all such quandles in which the images of the two quandles act trivially on each other.


Finally, we will have cause to consider a notion from Joyce \cite{JOYCE198237} which both gives a means of constructing quandles and a relationship between quandles and certain groups related to them.

\begin{definition}[Augmented Quandle]
An \textit{augmented quandle} $(Q,G)$ consists of a set $Q$, a group $G$ equipped with a right action on $Q$, and a map $|\cdot| : Q \to G$ satisfying
    \begin{enumerate}[label=\textbf{(AQ\arabic*)}]
        \item $q |q| = q$ for all $q \in Q$

        \item $|qg| = g^{-1}|q|g$ for all $q \in Q, g \in G$.
    \end{enumerate}
The map $|\cdot|$ is called the \textit{augmentation map}.
\end{definition}

Given an augmented quandle $(Q,G)$, we can define a quandle operation on $Q$ by $x \triangleleft y := x|y|$. If $Q$ was already endowed with a quandle structure, then the quandle structure prescribed by $|\cdot|$ is the same structure, hence why we can define an augmented quandle using just a set, as opposed to a quandle.

In \cite{JOYCE198237}, Joyce showed that, given a quandle $(Q,\triangleleft)$, the inclusion of generators into $\Adconj(Q)$ can be considered the universal augmentation of $Q$ in the sense that $\triangleleft$ induces a right $\Adconj(Q)$-action on $Q$ by quandle automorphisms such that the inclusion of generators is an augmentation. Furthermore, given a group $G$ and augmentation $\la\cdot\ra: Q \to G$, there is a unique group homomorphism $\varphi : \Adconj(Q) \to G$ such that $\varphi(|q|) = \la q\ra$ and $qg = q\varphi(g)$ for all $q \in Q, g \in \Adconj(Q)$. This is why $\Adconj(Q)$ is occasionally referred to as the universal augmentation group of $Q$.







\section{Profinite Quandles}

A profinite quandle is the inverse limit of a family of finite quandles organized into an inverse system. Hence profinite quandles are, in a sense, ``approximated'' by finite quandles. It is of interest to determine if properties that hold true for finite quandles hold true for their profinite counterparts.

\subsection{Inverse Systems and Profinite Quandles}

\begin{definition}
    A directed set $(\Lambda, \leq)$ is a pair consisting of a set $\Lambda$ and a partial order $\leq$ for which every pair of elements has an upper bound.
\end{definition}

\begin{definition}[Inverse system of quandles]
Let $(\Lambda,\leq)$ be a directed set. An inverse system of quandles and quandle homomorphisms over $\Lambda$ is a pair $(\{Q_\lambda\}_{\lambda\in\Lambda},\{\phi_{\a\b}\}_{\a\leq\b\in\Lambda})$ consisting of a family of quandles $\{Q_\lambda\}_{\lambda\in\Lambda}$ indexed by $\Lambda$ and a family of quandle homomorphisms $\varphi_{\a\b}:Q_\b\to Q_\a$ for all $\a\leq\b$ with $\a,\b\in\Lambda$ such that the following hold:
\begin{enumerate}[label=(\roman*)]
    \item The homomorphism $\varphi_{\a\a}$ is the identity on $Q_\a$ for all $\a\in\Lambda$.
    \item We have $\varphi_{\a\d}=\phi_{\a\b}\circ\phi_{\b\d}$ for all $\a,\b,\d\in\Lambda$ with $\a\leq\b\leq\d$.
\end{enumerate}
We call the maps $\varphi_{\a\b}$ the transition maps of the inverse system.
\end{definition}


\begin{definition}[Inverse limit]
Let $(\Lambda,\leq)$ be a directed poset and let $(\{Q_\lambda\}_{\lambda\in\Lambda},\{\phi_{\a\b}\}_{\a\leq\b\in\Lambda})$ be an inverse system of quandles and quandle homomorphisms over $\Lambda$. The inverse limit of this inverse system is given by
\[Q=\varprojlim_{\lambda\in\Lambda}Q_\lambda:=\l\{\vec{q}\in\prod_{\lambda\in\Lambda}Q_\lambda:q_\a=\varphi_{\a\b}(q_\b)\text{ for all }\a\leq\b\text{ in }\Lambda\r\}.\]
There are projection maps $\pi_\alpha:\varprojlim_{\lambda\in\Lambda}Q_\lambda\to Q_\a$ from the inverse limit to each of the factors $Q_\alpha$ for all $\a\in\Lambda$ such that
\[\begin{tikzcd}
	& {\displaystyle{\varprojlim_{\lambda \in \Lambda}} \hspace{1mm}Q_\lambda} \\
	{Q_\alpha} && {Q_\beta}
	\arrow["{\varphi_{\alpha\beta}}", from=2-3, to=2-1]
	\arrow["{\pi_\alpha}"', from=1-2, to=2-1]
	\arrow["{\pi_{\beta}}", from=1-2, to=2-3]
\end{tikzcd}\]
commutes for all $\alpha \leq \beta \in \Lambda$. The projection $\pi_\a$ is given by $\pi_\a(q)=q_\a$ for all $q=(q_\lambda)_{\lambda\in\Lambda}\in \varprojlim_{\lambda\in\Lambda}Q_\lambda$. Note that inverse limit coincides with the categorical notion of limit, or projective limit.
\end{definition}

\begin{proposition}
Let $(\Lambda,\leq)$ be a directed poset and let $(\{Q_\lambda\}_{\lambda\in\Lambda},\{\phi_{\a\b}\}_{\a\leq\b\in\Lambda})$ be an inverse system of quandles and quandle homomorphisms over $\Lambda$. The inverse limit $Q:=\varprojlim_{\lambda\in\Lambda}Q_\lambda$ is a quandle with operation given by coordinatewise operation. That is,
\[q\tl s=(q_\lambda)_{\lambda\in\Lambda}\tl(s_\lambda)_{\lambda\in\Lambda}:=(q_\lambda\tl_{\,Q_\lambda}s_\lambda)_{\lambda\in\Lambda}\]
for all $q,s\in Q$, where $\tl_{\,Q_\lambda}$ is the quandle operation on $Q_\lambda$ for all $\lambda\in\Lambda$.
\end{proposition}

\begin{proof}
First note that $\tl$ is closed. Indeed, if $q,s\in Q$, then $q\tl s=(q_\lambda\tl_{\,Q_\lambda}s_\lambda)_{\lambda\in\Lambda}$ and since each $\phi_{\a\b}$ is a quandle homomorphism we have that
\[\phi_{\a\b}(q_\b\tl_{\,Q_\b}s_\b)=\phi_{\a\b}(q_\b)\tl_{\,Q_\a}\phi_{\a\b}(s_\b)=q_\a\tl_{\,Q_\a}s_\a.\]

The remaining quandle axioms follow from those of the factors $Q_\lambda$.

(\textbf{Q1}) Let $q\in Q$. Observe that
\begin{align*}
q \tl q
&=(q_\lambda)_{\lambda\in\Lambda}\tl (q_\lambda)_{\lambda\in\Lambda}\\
&=(q_\lambda\tl_{\,Q_\lambda}q_\lambda)_{\lambda\in\Lambda}\\
&=(q_\lambda)_{\lambda\in\Lambda},\text{ by (\textbf{Q1}) for each $Q_\lambda$}\\
&=q.
\end{align*}

(\textbf{Q2}) Let $q,s\in Q$. Observe that
\begin{align*}
(q\tl s)\tl^{-1}s
&=((q_\lambda\tl_{\,Q_\lambda}s_\lambda)\tl_{\,Q_\lambda}^{-1}s_\lambda)_{\lambda\in\Lambda}\\
&=(q_\lambda)_{\lambda\in\Lambda},\text{ by (\textbf{Q2}) for each $Q_\lambda$}\\
&=q.
\end{align*}
Moreover, observe that
\begin{align*}
(q\tl ^{-1}s)\tl s
&=((q_\lambda\tl_{\,Q_\lambda}^{-1}s_\lambda)\tl_{\,Q_\lambda}s_\lambda)_{\lambda\in\Lambda}\\
&=(q_\lambda)_{\lambda\in\Lambda},\text{ by (\textbf{Q2}) for each $Q_\lambda$}\\
&=q.
\end{align*}

(\textbf{Q3}) Let $q,r,s\in Q$. Observe that
\begin{align*}
(q\tl r)\tl s
&=(q_\lambda\tl_{\,Q_\lambda}r_\lambda)\tl_{\,Q_\lambda}s_\lambda)_{\lambda\in\Lambda}\\
&=((q_\lambda\tl_{\,Q_\lambda}s_\lambda)\tl_{\,Q_\lambda}(r_\lambda\tl_{\,Q_\lambda}s_\lambda)\\
&=(q\tl s)\tl (r\tl s),
\end{align*}
where we have used (\textbf{Q3}) for each $Q_\lambda$.
\end{proof}

\begin{definition}[Profinite quandle]
A quandle $Q$ is profinite if it is the inverse limit of an inverse system of finite quandles and quandle homomorphisms.
\end{definition}

\begin{example}
All finite quandles are the inverse limit of the inverse system consisting of only the finite quandle itself and its identity map.
\end{example}

\begin{example}
    A quandle $Q$ is said to be {\em residually finite} if for every pair of elements $q \neq r \in Q$, there exists a finite quandle $Q^\prime$ and a quandle homomorphism $f:Q\rightarrow Q^\prime$ such that $f(q) \neq f(r)$.

    If $Q$ is a residually finite quandle, it maps by an injective quandle homomorphism into a profinite quandle $\widehat{Q}$, its profinite completion, which is the limit of the inverse system of all finite quotient quandles of $Q$.

    Residual finiteness turns out to be a property shared by many of the most studied quandles:  Bardakov, Singh and Singh \cite{Bardakov_2019} have shown that knot quandles and free quandles are residually finite.
\end{example}

\begin{remark}
    An inverse system of finite quandles can without loss of generatlity be assumed to have surjective maps. This is shown in ~\cite{Wilkes} in the case of an inverse system of finite sets, and the proof applies to an inverse system of quandles also.  We desire the quandle homomorphisms to be surjective since the inner automorphism functor, $\Inn$, is only a functor on the category of quandles and surjective quandle homomorphisms, not the category of quandles and quandle homomorphisms.  Hereinafter we assume without further comment that the homomorphisms in any inverse system of quandles are surjective.
\end{remark}

The notions of inverse system, inverse limit, and profiniteness are defined for groups similarly.

\begin{definition}[Inverse system of groups]
An inverse system of groups and group homomorphisms over a directed poset $(\Lambda,\leq)$ is a pair $(\{G_\lambda\}_{\lambda\in\Lambda},\{\phi_{\a\b}\}_{\a\leq\b\in\Lambda})$ consisting of a family of groups $\{G_\lambda\}_{\lambda\in\Lambda}$ indexed by $\Lambda$ and a family of group homomorphisms $\phi_{\a\b}:G_\b\to G_\a$ for all $\a\leq\b$ with $\a,\b\in\Lambda$ such that the following hold:
\begin{enumerate}[label=(\roman*)]
    \item The homomorphism $\varphi_{\a\a}$ is the identity on $G_\a$ for all $\a\in\Lambda$.
    \item We have $\varphi_{\a\d}=\phi_{\a\b}\circ\phi_{\b\d}$ for all $\a,\b,\d\in\Lambda$ with $\a\leq\b\leq\d$.
\end{enumerate}
The inverse limit of such an inverse system is given by
\[G=\varprojlim_{\lambda\in\Lambda}G_\lambda:=\l\{\vec{g}\in\prod_{\lambda\in\Lambda}G_\lambda:g_\a=\phi_{\a\b}(g_\b)\text{ for all }\a\leq\b\text{ in }\Lambda\r\}.\]
A group $G$ is profinite if it is the inverse limit of an inverse system of finite groups and group homomorphisms.
\end{definition}

Profinite groups have been studied extensively and have a well-developed theory.~\cite{wilson1998profinite} Because of the intimate relationships between groups and quandles, profinite groups will occur throughout our subsequent discussion of profinite quandle.  In the first instance, they provide us with more examples:

\begin{example}
The functor $\Conj: \textbf{Gp} \to \textbf{Qndl}$ admits a left adjoint $\Adconj: \textbf{Qndl} \to \textbf{Gp}$. Hence $\Conj$ preserves all limits. Therefore $\Conj(G)$ is a profinite quandle whenever $G$ is a profinite group.
\end{example}

\begin{example}
The functor $\Tak: \textbf{Ab} \to \textbf{Kei}$ admits a left adjoint $\text{Adtak}: \textbf{Kei} \to \textbf{Ab}$. Hence $\Tak$ preserves all limits. Therefore $\Tak(A)$ is a profinite kei whenever $A$ is a profinite abelian group.
\end{example}

To fill in these examples of profinite quandles we need examples of profinite groups:

\begin{example} \label{p-adics}
    For any prime $p$, the additive group of the $p$-adic integers, ${\mathbb Z}_p$, is a profinite abelian group.

    This is easy to see as it is the inverse limit of the totally ordered diagram \[ {\mathbb Z}/p \leftarrow {\mathbb Z}/p^2 \leftarrow {\mathbb Z}/p^3 \leftarrow \ldots \].
\end{example}

\begin{example}
    The profinite completion of the integers $\widehat{\mathbb Z}$ is a profinite abelian group.
\end{example}

\begin{example}
    The absolute galois group $G_K$ of any field $K$, that is the galois group of $K_{sep}$, the separable closure of $K$ is a profinite (usually non-abelian) group.  (As an aside, we recall that the separable closure of a perfect field is the same as its algebraic closure, and that imperfect fields must be of finite characteristic and positive transcendence degree over the prime field of the same characteristic.)
\end{example}

\begin{example}
    The automorphism group of any (infinite) rooted tree all of whose vertices are of finite degree is a profinite group.
\end{example}

\subsection{Topological Properties}

Profinite quandles, while the definition as given does not {\em a priori} involve a topology, are, in fact, examples of topological quandles.

\begin{definition}[Topological Quandle]
A \textit{topological quandle} $(Q,\tau,\triangleleft)$ is a topological space $(Q,\tau)$ that is also a quandle $(Q, \triangleleft)$ such that the quandle operations $\triangleleft: Q \times Q \to Q$ and $\triangleleft^{-1}: Q \times Q \to Q$ are continuous.
\end{definition}

Every profinite quandle $Q=\varprojlim_{\lambda\in\Lambda}Q_\lambda$ can be endowed with a Stone topology (compact, Hausdorff, totally disconnected) as a subspace of $\prod_{\lambda\in\Lambda}Q_\lambda$ equipped with the product topology, where the finite quandles are made into Stone spaces by giving them each the discrete topology. But any product of Hausdorff totally disconnected spaces is again Hausdorff totally disconnected, and by Tychonoff's Theorem, the arbitrary product of compact spaces is again compact. Hence $\prod_{\lambda\in\Lambda}Q_\lambda$ is a Stone space. 

Subspaces of Hausdorff totally disconnected spaces are again Hausdorff totally disconnected, and by~\cite[Lemma 1.1.2]{wilson1998profinite} we have $\varprojlim Q_\lambda \subseteq \prod Q_\lambda$ is closed, hence compact. Therefore $Q$ is a Stone space.

The same argument established the well-known result that every profinite group $G=\varprojlim_{\lambda\in\Lambda}G_\lambda$ can be endowed with a Stone topology as a subspace of the product topology on $\prod_{\lambda\in\Lambda}G_\lambda$. It is also known that, conversely, every Stone topological group is a profinite group \cite{johnstone1982stone}. 

The corresponding converse result does not hold for quandles, as the following example due to Ariel Davis \cite{Dav24} shows:

\begin{example}
    Consider the space $\mathcal{Z} := \mathbb{Z}_* \coprod \{s\}$, where $\mathbb{Z}_*$ is the one-point compactification of the discrete space of integers.
    Plainly $\mathcal{Z}$ is a Stone space.  Now endow it with a continuous quandle operation by defining $x\triangleleft y = x$ unless $x\in \mathbb{Z} \subset \mathcal{Z}$ and $y=s$, in which case $x\triangleleft s = x+1$.

    It is easy to see that the operation is continuous and satisfies the quandle axioms, so that $\mathcal Z$ is a Stone topological quandle.

    However, $\mathcal{Z}$ is not profinite.  If it were, it would have to be an inverse limit of finite quotient quandles each with the discrete topology.  It is easy to classify the finite quotients of the underlying (non-topological) quandle of $\mathcal{Z}$:  they are the one and two element trivial quandles (in the latter case the quotient map maps $s$ to one element, and all of $\mathbb{Z}_*$ to the other) and quandles with underlying sets $\mathbb{Z}/n\mathbb{Z} \coprod \{\infty\} \coprod \{s\}$, with the operation given by first projection unless $[x]\in \mathbb{Z}/n\mathbb{Z}$, in which case $[x]\triangleleft s = [x+1]$.  Here $[x]$ denotes the residue of $x$ modulo $n$.  
    
    But the projection maps from a profinite quandle with its induced Stone topology to the discrete finite quandles in its defining diagram are continuous.  The quotient maps from $\mathcal{Z}$ to its finite quotients of  the form $\mathbb{Z}/n\mathbb{Z} \coprod \{\infty\} \coprod \{s\}$ are not continuous, since the inverse image of the singleton $\{\infty\}$ is $\{\infty\}$. The set $\{\infty\}$ is not open since the open neighborhoods of $\infty$ in $\mathcal{Z}$ are all a union of a cofinite subset of $\mathbb{Z}$ with either $\{\infty\}$ or $\{\infty, s\}$.  However, such a projection map would need to be among those to the defining diagram were $\mathcal{Z}$ to be profinite, as the projections to the defining diagram collectively separate element, and the maps to the trivial quotients do not.

    Thus $\mathcal{Z}$ is a non-profinite Stone topological quandle.
\end{example}

Throughout this paper, the existence of the Stone topology on profinite quandles will be important in a number of proofs. 

\begin{proposition}
Let $Q$ be a profinite quandle with inverse system $(\{Q_\lambda\}_{\lambda \in \Lambda}, \{\varphi_{\alpha\beta}\}_{\alpha \leq \beta \in \Lambda})$ and projection maps $\pi_\lambda : Q \to Q_\lambda.$ The collection of subsets

\[\mc{B} := \{\pi_\lambda^{-1}(x) : x \in Q_\lambda, \lambda \in \Lambda\}\]

forms a basis for a Stone topology on $Q.$ We call $\mc{B}$ the ``slim basis'' of $Q$ and its elements ``slim basic open sets''.
\end{proposition}

\begin{proof}
We will directly verify that $\mc{B}$ is a basis.
\begin{enumerate}
    \item \textbf{Covering Property:} Let $q=(q_\lambda)_{\lambda\in\Lambda}\in Q$. We have that $q \in \pi_\lambda^{-1}(q_\lambda)$ for each $\lambda \in \Lambda$.
    \item \textbf{Covering of Finite Intersections:} 
    Consider $\pi_\a^{-1}(x)$ and $\pi_\b^{-1}(y)$ for some elements $x\in Q_\a$ and $y\in Q_\b$. If the intersection is empty, there is nothing to show, so suppose $q = (q_\lambda)_{\lambda \in \Lambda} \in \pi_\a^{-1}(x) \cap \pi_\b^{-1}(y)$.  So $\pi_\a(q) = q_\a = x$ and $\pi_\b(q) = q_\b = y$.
    But $\Lambda$ is directed, so there exits $\gamma$ with $\a \leq \gamma$ and $\b \leq \gamma$.  

    Now, by the construction of the inverse limit, $\phi_{\a \gamma}(q_\gamma) = q_\a = x$ and $\phi_{\b \gamma}(q_\gamma) = q_\b = y$.  In fact, for the same reason $\pi_\gamma^{-1}(q_\gamma) \subset \pi_\a^{-1}(x) \cap \pi_\b^{-1}(y)$.  But $q = (q_\lambda)_{\lambda \in \Lambda} \in \pi_\gamma^{-1}(q_\gamma)$.  So, as required, given any element of an intersection of basic opens, there is a basic open neighborhood contained in the intersection.
\end{enumerate}

Thus $\mc{B} := \{\pi_\lambda^{-1}(x) : x \in Q_\lambda, \lambda \in \Lambda\}$ is a basis for the topology on $Q$.  
\end{proof}

The same argument works for profinite groups.  It will generally be easiest to verify that maps to a profinite quandle are continuous by verifying that the inverse images of slim basic open sets are open.

\subsection{Products and Disjoint Unions}

\begin{proposition} \label{binaryproducts}
If $Q$ and $S$ are profinite quandles, then the product $Q \times S$ is a profinite quandle.\footnote{This result was observed without proof by Singh \cite{S23}.}
\end{proposition}

\begin{proof}
Let $(\{Q_\d\}_{\d\in\Delta},\{\phi_{\a\b}\}_{\a\leq\b\in\Delta})$ and $(\{S_\lambda\}_{\lambda\in\Lambda},\{\psi_{\a\b}\}_{\a\leq\b\in\Delta})$ be inverse systems for $Q$ and $S$ over directed posets $\Delta$ and $\Lambda$, respectively, and denote the projections $\pi^Q_\delta: Q \to Q_\delta$ and $\pi^S_\lambda \to S \to S_\lambda$. Consider the collection $\{Q_\delta \times S_\lambda\}_{(\delta,\lambda) \in \Delta \times \Lambda}$. Define a collection of quandle homomorphisms
\[\left\{\sigma_{(\alpha,\xi),(\beta,\eta)} : Q_\b \times S_\eta \to Q_\a \times S_\xi \mid (\alpha,\xi),(\beta,\eta) \in \Delta \times \Lambda, (\alpha,\xi)\leq (\beta,\eta)\right\}\]

where $\sigma_{(\alpha,\xi),(\beta,\eta)} := \varphi_{\a\b} \times \psi_{\xi\eta}$ and $\Delta \times \Lambda$ is endowed with the product order. It is clear that $(\{Q_\delta \times S_\lambda\}_{(\delta,\lambda) \in \Delta \times \Lambda}, \{\sigma_{(\alpha,\xi),(\beta,\eta)}\}_{(\alpha,\xi)\leq (\beta,\eta) \in \Delta \times \Lambda})$ forms an inverse system, where $\rho_{(\a,\xi)} : \varprojlim Q_\delta \times S_\lambda \to Q_\a \times S_\xi$ denotes the projection maps. There are maps from $Q \times S$ to each $Q_\delta \times S_\lambda$, namely $\pi_\delta^Q \times \pi_\lambda^S$, so by the universal property of limits, there is a unique continuous quandle homomorphism $\Psi : Q \times S \to \varprojlim Q_\delta \times S_\lambda$ such that the following diagram commutes:

\[\begin{tikzcd}
	& {Q \times S} \\
	& {\varprojlim Q_\delta \times S_\lambda} \\
	{Q_\a \times S_\xi} && {Q_\b \times S_\eta}
	\arrow["{\rho_{(\b,\eta)}}"{description}, from=2-2, to=3-3]
	\arrow["{\rho_{(\a,\xi)}}"{description}, from=2-2, to=3-1]
    \arrow["{\sigma_{((\a,\xi),(\b,\eta))}}"', from=3-3, to=3-1]	
    \arrow["{\pi^Q_\b \times \pi^S_\eta}", from=1-2, to=3-3]
	\arrow["{\pi^Q_\a \times \pi^S_\xi}"', from=1-2, to=3-1]
	\arrow["{\exists!\Psi}", dashed, from=1-2, to=2-2]
\end{tikzcd}\]

We claim $\Psi$ is an isomorphism. To see that $\Psi$ is injective, let $x,y \in Q \times S$ such that $\Psi(x) = \Psi(y).$ Then for all $\delta \in\Delta$ and $\lambda \in \Lambda,$
\[\rho_{(\delta,\lambda)}(\Psi(x)) = \rho_{(\delta,\lambda)}(\Psi(y)) \Rightarrow \pi^Q_\delta\times\pi^S_\lambda(x) = \pi^Q_\delta\times\pi^S_\lambda(y) \Rightarrow x = y.\]

For surjectivity the quandle $Q \times S$ is a product of a compact spaces, hence $Q \times S$ is compact. Furthermore $\Psi(Q \times S) \subseteq \varprojlim Q_i \times S_i$ is compact and therefore closed since $\varprojlim Q_\delta \times S_\lambda$ is Hausdorff. Thus it suffices to show that $\Psi(Q \times S)$ is dense.

Let $x \in \varprojlim Q_\delta \times S_\lambda$ and let $y \in Q_\delta \times S_\lambda$ such that $\rho_{(\delta,\lambda)}^{-1}(y)$ is a neighborhood of $x$. We have $\rho_{(\delta,\lambda)}(x) = y,$ so let $z \in (\pi^Q_\delta \times \pi^S_\lambda)^{-1}(y).$ Then $\rho_{(\delta,\lambda)}(\Psi(z)) = y,$ meaning $\Psi(z) \in \rho_{(\delta,\lambda)}^{-1}(y),$ or $\rho_{(\delta,\lambda)}^{-1}(y) \cap (\Psi(Q \times S) \setminus \{x\}) \neq \varnothing$, so $x$ is a limit point of $\Psi(Q \times S)$. Thus, $\Psi(Q \times S) \subseteq \varprojlim Q_\delta \times S_\lambda$ is dense and 
\[Q \times S \cong \varprojlim Q_\delta \times S_\lambda.\] 
\end{proof}

The same indexing by the products of the two directed sets underlying the inverse systems as was used to prove 
Proposition \ref{binaryproducts}, but using pairwise disjoint unions of finite quandles rather than pairwise products let us show the following.

\begin{proposition}
If $Q$ and $S$ are profinite quandles, then the disjoint union $Q\sqcup S$ is a profinite quandle.
\end{proposition}

\begin{proof}
Let $Q=\varprojlim_{\d\in\Delta}Q_\d$ and $S=\varprojlim_{\lambda\in\Lambda}S_\lambda$ be profinite quandles with inverse systems $(\{Q_\d\}_{\d\in\Delta},\{\phi_{\a\b}\}_{\a\leq\b\in\Delta})$ and $(\{S_\lambda\}_{\lambda\in\Lambda},\{\psi_{\zeta\eta}\}_{\zeta\leq\eta\in\Lambda})$ over directed posets $\Delta$ and $\Lambda$, respectively. Consider the disjoint unions $Q_\d\sqcup S_\lambda$ of the factors of $Q$ and $S$. We can index the $Q_\d\sqcup S_\lambda$ by the directed poset $\Delta\times\Lambda$ with the product order. To construct an inverse system out of $\{Q_\d\sqcup S_\lambda\}_{(\d,\lambda)\in\Delta\times\Lambda}$, we define a collection of quandle homomorphisms
\[\{\sigma_{(\a,\zeta),(\b,\eta)}:Q_\b\sqcup S_\eta\to Q_\a\sqcup S_\zeta\}_{(\a,\zeta)\leq(\b,\eta)\in\Delta\times\Lambda},\]
where $\sigma_{(\a,\zeta),(\b,\eta)}:=\phi_{\a\b}\sqcup\psi_{\zeta\eta}$. That each $\sigma_{(\a,\zeta),(\b,\eta)}$ are quandle homomorphisms follows from the fact that each $\phi_{\a\b}$ and $\psi_{\zeta\eta}$ are quandle homomorphisms, along with the identity maps $\Id_{Q_\b},\Id_{S_\eta}$. To show that $(\{Q_\d\sqcup S_\lambda\}_{(\d,\lambda)\in\Delta\time\Lambda},\{\sigma_{(\a,\zeta),(\b,\eta)}\}_{(\a,\zeta)\leq(\b,\eta)\in\Delta\times\Lambda})$ is an inverse system it remains to observe that if $(\a,\zeta)\leq(\b,\eta)\leq(\gamma,\e)$, then
\[\sigma_{(\a,\zeta),(\gamma,\e)}=\sigma_{(\a,\zeta),(\b,\eta)}\circ\sigma_{(\b,\eta),(\gamma,\e)}.\]
To see this, let $v\in Q_\gamma\sqcup S_\e$. Without loss of generality suppose $v\in Q_\gamma$. Then
\begin{align*}
(\sigma_{(\a,\zeta),(\b,\eta)}\circ\sigma_{(\b,\eta),(\gamma,\e)})(v)
&=\sigma_{(\a,\zeta),(\b,\eta)}(\sigma_{(\b,\eta),(\gamma,\e)}(v))\\
&=\sigma_{(\a,\zeta),(\b,\eta)}((\phi_{\b\gamma}\sqcup\psi_{\eta\e})(v))\\
&=\sigma_{(\a,\zeta),(\b,\eta)}(\phi_{\b\gamma}(v))\\
&=\phi_{\a\b}(\phi_{\b\gamma}(v))\\
&=(\phi_{\a\b}\circ\phi_{\b\gamma})(v)\\
&=\phi_{\a\gamma}(v),\text{ since }\phi_{\a\gamma}=\phi_{\a\b}\circ\phi_{\b\gamma}\\
&=(\phi_{\a\gamma}\sqcup\psi_{\zeta\e})(v)\\
&=\sigma_{(\a,\zeta),(\gamma,\e)}(v).
\end{align*}
If $v\in S_\e$, then $\sigma_{(\a,\zeta),(\gamma,\e)}(v)=(\sigma_{(\a,\zeta),(\beta,\eta)}\circ\sigma_{(\b,\eta),(\gamma,\e)})(v)$ by the same argument on the $\psi_{\zeta\eta}$. Hence $(\{Q_\d\sqcup S_\lambda\}_{(\d,\lambda)\in\Delta\times\Lambda},\{\sigma_{(\a,\zeta),(\b,\eta)}\}_{(\a,\zeta)\leq(\b,\eta)\in\Delta\times\Lambda})$ is an inverse system, whose inverse limit is given by
\begingroup\makeatletter\def\f@size{8}\check@mathfonts
\[\varprojlim_{(\d,\lambda)\in\Delta\times\Lambda}Q_\d\sqcup S_\lambda=\l\{\vec{v}\in\prod_{(\d,\lambda)\in\Delta\times\Lambda}Q_\d\sqcup S_\lambda:v_{(\a,\zeta)}=\sigma_{(\a,\zeta),(\beta,\eta)}(v_{(\b,\eta)})\text{ for all }(\a,\zeta)\leq(\b,\eta)\r\}.\]
\endgroup
Consider the map
\[\Psi:Q\sqcup S\to\varprojlim_{(\d,\lambda)\in\Delta\times\Lambda}Q_\d\sqcup S_\lambda\]
given by
\[\Psi(v):=\begin{cases}(v_\d)_{(\d,\lambda)\in\Delta\times\Lambda},&\text{if }v=(v_\d)_{\d\in\Delta}\in Q;\\(v_\lambda)_{(\d,\lambda)\in\Delta\times\Lambda},&\text{if }v=(v_\lambda)_{\lambda\in\Lambda}\in S.\end{cases}\]
We claim that $\Psi$ is a quandle isomorphism. To see that $\Psi$ is injective, let $u,v\in Q\sqcup S$ with $u\neq v$. If $u\in Q$ and $v\in S$, then it is clear that $\Psi(v)\neq\Psi(v)$. Hence, we only need to consider the cases $u,v\in Q$ and $u,v\in S$. Suppose $u,v\in Q$. Since $Q=\varprojlim_{\d\in\Delta}Q_\d$, we may write $u=(u_\d)_{\d\in\Delta}$ and $v=(v_\d)_{\d\in\Delta}$. Since $u\neq v$, there exists $\a\in\Delta$ for which $u_\a\neq v_\a$. It follows that $\Psi(u)=(u_\d)_{(\d,\lambda)\in\Delta\times\Lambda}$ and $\Psi(v)=(v_\d)_{(\d,\lambda)\in\Delta\times\Lambda}$ differ in all coordinates $(\a,\lambda)\in\Delta\times\Lambda$, so that $\Psi(u)\neq\Psi(v)$. The argument for injectivity in the case $u,v\in S$ is the same. To see that $\Psi$ is surjective, consider an element $\vec{v}\in\varprojlim_{(\d,\lambda)\in\Delta\times\Lambda}Q_\d\sqcup S_\lambda$. Since
\[v_{(\a,\zeta)}=\sigma_{(\a,\zeta),(\b,\eta)}(v_{(\b,\eta)})=(\phi_{\a\b}\sqcup\psi_{\zeta\eta})(v_{(\b,\eta)})\]
for all $(\a,\zeta)\leq(\b,\eta)$, it follows from the least upper bound property of directed posets that if some entry $v_{(\a,\zeta)}$ of $\vec{v}$ lies in $Q_\a$ (resp., in $S_\zeta$), then all entries of $\vec{v}$ lie in one of the $Q_\d$ (resp., one of the $S_\lambda$). That is, the entries of an element $\vec{v}\in\varprojlim_{(\d,\lambda)\in\Delta\times\Lambda}Q_\d\sqcup S_\lambda$ either lie entirely in the components $Q_\d$ of $Q$ or the components $S_\lambda$ of $S$. Moreover, if $\vec{v}$ has entries entirely in the $Q_\d$ and if $(\a,\zeta)\leq(\a,\eta)$, with $\zeta\neq\eta$, then $v_{(\a,\zeta)}=v_{(\a,\eta)}$ since $\sigma_{(\a,\zeta),(\a,\eta)}=\phi_{\a\a}\sqcup\psi_{\zeta\eta}$, and similarly if $\vec{v}$ has entries entirely in the $S_\lambda$. Hence, every element $\vec{v}=(v_{(\d,\lambda)})_{(\d,\lambda)\in\Delta\times\Lambda}\in\varprojlim_{(\d,\lambda)\in\Delta\times\Lambda}Q_\d\sqcup S_\lambda$ can be written as either $(v_\d)_{(\d,\lambda)\in\Delta\times\Lambda}$ or $(v_\lambda)_{(\d,\lambda)\in\Delta\times\Lambda}$. This shows that $\Psi$ is surjective.

Finally, we need to show that $\Psi$ is a quandle homomorphism. We denote the quandle operation of $Q\sqcup S$ by $\tl_{\,\sqcup}$ and the quandle operation of $\varprojlim_{(\d,\lambda)\in\Delta\times\Lambda}Q_\d\sqcup S_\lambda$ by $\tl_{\,\lim}$. The quandle operations of each of the factors $Q_\d$ and $S_\lambda$ are denoted $\tl_{\,Q_\d}$ and $\tl_{\,S_\lambda}$, respectively. Let $u,v\in Q\sqcup S$. We break into cases:

\ul{\textbf{Case 1:}} $u,v\in Q$ or $u,v\in S$.

Suppose $u,v\in Q$. Observe that
\begin{align*}
\Psi(u\tl_{\,\sqcup}v)
&=\Psi(u\tl_{\,Q}v)\\
&=\Psi((u_\d\tl_{\,Q_\d}v_\d)_{\d\in\Delta})\\
&=(u_\d\tl_{\,Q_\d}v_\d)_{(\d,\lambda)\in\Delta\times\Lambda}\\
&=(u_\d\tl_{\,Q_\d\sqcup S_\lambda}v_\d)_{(\d,\lambda)\in\Delta\times\Lambda}\\
&=(u_\d)_{(\d,\lambda)\in\Delta\times\Lambda}\tl_{\,\lim}(v_\d)_{(\d,\lambda)\in\Delta\times\Lambda}\\
&=\Psi(u)\tl_{\,\lim}\Psi(v).
\end{align*}
The argument is the same when $u,v\in S$.

\ul{\textbf{Case 2:}} $u\in Q$ and $v\in S$, or $u\in S$ and $v\in Q$.

Suppose $u\in Q$ and $v\in S$. Observe that
\begin{align*}
\Psi(u\tl_{\,\sqcup}v)
&=\Psi(u)\\
&=(u_\d)_{(\d,\lambda)\in\Delta\times\Lambda}\\
&=(u_\d\tl_{\,Q_\d\sqcup S_\lambda}v_\lambda)_{(\d,\lambda)\in\Delta\times\Lambda}\\
&=(u_\d)_{(\d,\lambda)\in\Delta\times\Lambda}\tl_{\,\lim}(v_\lambda)_{(\d,\lambda)\in\Delta\times\Lambda}\\
&=\Psi(u)\tl_{\,\lim}\Psi(v).
\end{align*}
The argument is the same when $u\in S$ and $v\in Q$.

Since both cases exhaust all possibilities for $u,v\in Q\sqcup S$, it follows that $\Psi$ is a quandle homomorphism. Since $\Psi$ is also bijective, it follows that $\Psi$ is a quandle isomorphism. Hence $Q\sqcup S$ is profinite.
\end{proof}

Proposition \ref{binaryproducts} in fact follows from a more general theorem, the proof of which is more involved.  We stated and proved the case of finite products separately because of the strong parallels between its proof as given and the proof of the preservation of profiniteness under finite disjoint unions.

\begin{theorem} \label{arbitraryproducts}
    Let $\{ Q_i \;|\; i\in \mathcal{I} \}$ be a set of profinite quandles, then $\prod_{i\in \mathcal{I}}Q_i$ with coordinatewise operations is a profinite quandle.
\end{theorem}

\begin{proof}
    Let $Q_i = \varprojlim_{\lambda \in \Lambda_i} Q_\lambda$  for each $i\in \mathcal{I}$. Here we suppress mention of the maps $\phi_{\lambda,\mu}:Q_\lambda \rightarrow Q_\mu$ for $\mu \leq \lambda \in \Lambda_i$, though we will have cause to mention them later. We proceed by showing that $\prod_{i\in \mathcal{I}}Q_i$ can also be realized as the limit of an inverse system of finite quandles over a directed set built from the $\Lambda_i$'s and the finite subsets of $\mathcal{I}$ ordered by reverse inclusion.

    The underlying set of this directed set is

    \[ \Lambda_\mathcal{I} := \coprod_{\{J\subseteq \mathcal{I} \; | \; 0 < |J| < \infty \}} \prod_{i\in J} \Lambda_i .\]

    We denote elements of this disjoint union of products by $(J,\langle \lambda_i \rangle_{i\in J})$, where the elements of the $J$-indexed tuple lie in the directed set corresponding to the index, that is $\lambda_i \in \Lambda_i$.

    This set is then equipped with a partial order given by $(J,\langle \lambda_i \rangle_{i\in J}) \leq (K,\langle \mu_i \rangle_{i\in K}) $ precisely when $J\subseteq K$ and for all $i\in J$, $\lambda_i \leq \mu_i$ in the partial ordering which makes $\Lambda_i$ a directed set.  Reflexivity and anti-symmetry are trivial and transitivity is easy to check.

    To see that $(\Lambda_\mathcal{I},\leq)$ is a directed set, consider $(J,\langle \lambda_i \rangle_{i\in J})$ and $(K,\langle \mu_i \rangle_{i\in K})$.  Then $(J\cup K, \langle \nu_i \rangle_{i\in J\cup K})$ where $\nu_i = \lambda_i$ if $i\in J\setminus K$, $\nu_i = \mu_i$ if $i \in K\setminus J$ and $\nu_i \in \Lambda_i$ is chosen so that $\lambda_i \leq \nu_i$ and $\mu_i \leq \nu_i$ in $\Lambda_i$ by the directedness of $\Lambda_i$ whenever $i\in J\cap K$, plainly provides the needed element greater than or equal to both $(J,\langle \lambda_i \rangle_{i\in J})$ and $(K,\langle \mu_i \rangle_{i\in K})$.

    Now, consider the inverse system of finite quandles indexed by $\Lambda_\mathcal{I}$, in which

    \[ Q_{(J,\langle \lambda_i \rangle_{i\in J})} := \prod_{i\in J} Q_{\lambda_i}\]

    \noindent and the projection $\phi_{(J,\langle \lambda_i \rangle),(K,\langle \mu_i\rangle)}$ is given by the projection $\pi: \prod_J Q_{\lambda_i}\rightarrow \prod_K Q_{\lambda_i}$ onto the coordinates lying in $K$, followed by $\prod_K \phi_{\lambda_i,\mu_i}$, the product of the quandle homomorphisms in the diagrams over the $\Lambda_i$'s for $i\in K$ whose limits are the original $Q_i$'s.

    We claim that the product $\prod_{i\in \mathcal{I}}Q_i$ is the limit of this inverse system.  We proceed by first showing that the same data determines the elements of both:

    An element of $\prod_{i\in \mathcal{I}}Q_i$ is in the first instance an $\mathcal{I}$-tuple of elements of the $Q_i$'s for $i\in \mathcal{I}$, but an element of $Q_i$ is a $\Lambda_i$-tuple $\langle q_\lambda \rangle_{\lambda \in \Lambda_i} $ with $q_\lambda \in Q_\lambda$ of elements of the finite quandles $Q_\lambda$ $\lambda\in \Lambda_i$ satisfying the compatibility condition $\phi_{\lambda,\mu}(q_\mu) = q_\lambda$.

    Now, observe that this is the same data that gives the portion of the $\Lambda_\mathcal{I}$-tuple (satisfying a corresponding compatibility condition) in entries where the set in the index is a singleton.  But the compatibility condition implies that the rest of the entries are completely determined by those with a singleton as the indexing set.

    It remains only to show that the topologies induced by the two constructions agree.  Elements of the product of the profinite quandles are those elements of $\prod_{i\in\mathcal{I}} \prod_{\lambda\in \Lambda_i} Q_i$, with the entries in the $\Lambda_i$-tuples satisfying the compatibility condition for each $i \in \mathcal{I}$. We may identify them with the coordinates of elements in 
    $\varprojlim_{\lambda \in \Lambda_i} Q_{(J,\langle\lambda_i \rangle_\in J)}$ in which $J$ is a singleton.  The argument above showed that the restriction of the projection from $\prod_\mathcal{I} \prod_{\Lambda_i} Q_{\lambda_i}$ to those coordinates gives a bijection between the limit we claim witnesses the profiniteness of the product and the product itself.  But projections are continuous open maps, and therefore the restriction is a homeomorphism.
\end{proof}


\subsection{Subquandles of Profinite Quandles}

Given a profinite quandle, it is natural to inquire as to whether or not its subquandles are also profinite. We provide a necessary and equivalent condition for when subquandles of a profinite quandle are profinite.

It is easy to see that in general subquandles of profinite quandles need not themselves be profinite:  any infinite residually finite quandle (with the discrete toplogy) may be identified with a proper non-profinite subquandle of its profinite completion, as for example $\Tak({\mathbb Z}) \prec \Tak(\hat{\mathbb Z}).$

In giving a characterization of those subquandles of a profinite quandle which are themselves profinite, it will be useful to have a characterization of dense subquandles.

\begin{proposition} \label{density}
    If $Q$ is a profinite quandle with inverse system $\{\{ Q_\lambda \}_{\lambda \in \Lambda}, \{ \varphi_{\alpha\beta} \}_{\alpha \leq \beta \in \Lambda}\}$ and $S \preceq Q$, then $S$ is dense if and only if the restrictions of the projection homomorphisms $\pi_\lambda:Q\rightarrow Q_\lambda$ to $S$ are surjective.
\end{proposition}

 \begin{proof} 
Recall that a subspace being dense is equivalent to having a non-empty intersection with every non-empty open set, which is in turn equivalent to having a non-empty intersection with every open set in a basis.  Applying this observation to the slim basis, we see that $S$ is dense if and only if it intersects every $\pi_\lambda^{-1}(x)$. This is equivalent to every $x\in Q_\lambda$ having a nonempty preimage under $\pi_\lambda |_S$ for all $\lambda\in \Lambda$, thus establishing the proposition.
 \end{proof}



\begin{proposition} \label{profinitesubisclosed}
Let $Q$ be a profinite quandle. A subquandle $S\preceq Q$ is profinite if and only if $S$ is closed in $Q$.
\footnote{This result was independently obtained by Singh \cite{S23} Proposition 4.8.}
\end{proposition}

\begin{proof}
Let $(\{Q_\lambda\}_{\lambda\in\Lambda},\{\phi_{\a\b}\}_{\a\leq\b\in\Lambda})$ be an inverse system for $Q$ over a directed poset $\Lambda$ and denote the projections $\pi_\lambda:Q\to Q_\lambda$.

$(\Rightarrow)$ Suppose $S\preceq Q$ is profinite. This implies that $S$ is compact. Since $Q$ is Hausdorff, this implies that $S$ is closed in $Q$.

$(\Leftarrow)$ Suppose $S$ is closed in $Q$. We can restrict $\varphi_{\alpha\beta}$ to $\pi_\beta(S)$, denoted $\hat\varphi_{\alpha\beta} := \varphi_{\alpha\beta}|_{\pi_\beta(S)},$ which forms an inverse system $(\{\pi_\lambda(S)\}_{\lambda \in \Lambda},\{\hat\varphi_{\alpha\beta}\}_{\alpha \leq \beta \in \Lambda})$ with projections $\rho_\alpha : \varprojlim \pi_\alpha(S) \to \pi_\alpha(S)$. There are maps from $S$ to each $\pi_\alpha(S),$ so by the universal property of limits, there is a unique continuous quandle homomorphism $\Psi: S \to \varprojlim \pi_\alpha(S).$

\[\begin{tikzcd}
	& {S} \\
	& {\varprojlim \pi_\alpha(S)} \\
	{\pi_\alpha(S)} && {\pi_\beta(S)}
	\arrow["{\rho_\beta}"{description}, from=2-2, to=3-3]
	\arrow["{\rho_\alpha}"{description}, from=2-2, to=3-1]
	\arrow["{\hat\varphi_{\alpha\beta}}"', from=3-3, to=3-1]
	\arrow["{\pi_\beta}", from=1-2, to=3-3]
	\arrow["{\pi_\alpha}"', from=1-2, to=3-1]
	\arrow["{\exists!\Psi}", dashed, from=1-2, to=2-2]
\end{tikzcd}\]

We claim $\Psi$ is an isomorphism. To see that $\Psi$ is injective, let $q_1,q_2 \in S$ such that $\Psi(q_1) = \Psi(q_2).$ Then for all $\alpha \in \Lambda,$
\[\rho_\alpha(\Psi(q_1)) = \rho_\alpha(\Psi(q_2)) \Rightarrow \pi_\alpha(q_1) = \pi_\alpha(q_2) \Rightarrow q_1 = q_2.\]

For surjectivity, $S$ is a closed subspace of a compact space, hence $S$ is compact. Furthermore, $\Psi(S) \subseteq \varprojlim \pi_\alpha(S)$ is compact and therefore closed since $\varprojlim \pi_\alpha(S)$ is Hausdorff. But by construction all of the maps $\pi_\alpha:S\rightarrow \pi_\alpha(S)$ are surjective, thus by Proposition \ref{density} $S$ is dense in the profinite $\varprojlim \pi_\alpha(S)$ and we are done.


\end{proof}

An identical proof to that of Proposition \ref{profinitesubgroupisclosed}, excepting that the word ``quandle'' replaced with ''group'' throughout, gives us

\begin{proposition} \label{profinitesubgroupisclosed}
Let $G$ be a profinite group. A subgroup $S < G$ is profinite if and only if $S$ is closed in $G$.
\end{proposition}

We will have occasion to use this in one of our examples below.

Topological closure gives a way of producing profinite subquandles from arbitrary subquandles due the the following result about general topological quandles.

\begin{proposition}
Let $Q$ be a topological quandle and $Q' \preceq Q$ be a subquandle. Then $\overline{Q'},$ the topological closure of $Q',$ is a subquandle of $Q$.
\end{proposition}

\begin{proof}
Let $x,y \in \overline{Q'}$ and suppose $x \triangleleft y \not\in \overline{Q'}.$ We have $\overline{Q'}^c$ is open, so there is a neighborhood $U \subseteq Q \times Q$ of $(x,y)$ such that $u \triangleleft v \not\in \overline{Q'}$ for every $(u,v) \in U.$ Furthermore, there exists neighborhoods $V_x,V_y \subseteq Q$ of $x$ and $y,$ respectively, such that $V_x \times V_y \subseteq U.$ Let $u \in V_x \cap Q'$ and $v \in V_y \cap Q'$. Then $u \triangleleft v \in Q' \subseteq \overline{Q'}$, but $(u,v) \in U$ by construction, a contradiction.
\end{proof}

\section{Profinite Analogs of Inn(Q)} 

It is fairly easy to show that the inner automorphism group of a profinite group is always profinite, being the quotient of the profinite group by its center which is necessarily a closed normal subgroup. However, as Counterexample \ref{nonprofiniteInn} shows, the inner automorphism group of a profinite quandle need not be a profinite group. Therefore, we introduce profinite analogs of the inner automorphism group (i.e. profinite groups that act continously on the profinite quandle in a way analogous to the action of the inner automorphism group).

\begin{definition}[Profinite inner automorphism group]
Let $Q$ be a profinite quandle with inverse system $(\{Q_\lambda\}_{\lambda \in \Lambda},\{\varphi_{\a\b}\}_{\a\leq\b\in\Lambda})$. The functor $\Inn$ provides an inverse system over $\{\Inn(Q_\lambda)\}_{\lambda\in\Lambda}$, so define the \textit{profinite inner automorphism group over D} to be $\widehat{\Inn}_D(Q):=\varprojlim\Inn(Q_\a)$, where $D$ refers to the diagram of the prescribed inverse system.
\end{definition}

This construction explicitly depends on the diagram used to witness profiniteness; given two filtered diagrams $D$ and $D'$ under a profinite quandle $Q$, $\pInn_D(Q)$ and $\pInn_{D'}(Q)$ are not necessarily isomorphic. While this presents an obvious disadvantage, $\pInn_D(Q)$ still provides a description of a $Q$ in terms of the inner automorphism groups of the component quandles while also being profinite. Furthermore, there are properties of $\pInn_D(Q)$ that arise independent of the choice of diagram.

\begin{proposition}\label{Inn_dense}
Let $Q$ be a profinite quandle with inverse system $(\{Q_\lambda\}_{\lambda \in \Lambda},\{\varphi_{\a\b}\}_{\a \leq \b \in \Lambda})$. Then $\Inn(Q)$ is isomorphic to a dense subgroup of $\pInn_D(Q).$
\end{proposition}

\begin{proof}
Denote projections $\pi_\lambda : \pInn_D(Q) \to \Inn(Q_\lambda).$ For an inner automorphism $S_x \in \Inn(Q)$ where $q \mapsto q \triangleleft x,$ there are induced inner automorphisms $S_{x_\lambda} \in \Inn(Q_\lambda)$ by $q \mapsto q \triangleleft x_\lambda$ where $\pi_\lambda(x) = x_\lambda$ for each $\lambda \in \Lambda,$ so there are induced maps $\rho_\lambda : \Inn(Q) \to \Inn(Q_\lambda)$ where $S_x \mapsto S_{x_\lambda}.$ Then there is a unique homomorphism $\Psi: \Inn(Q) \to \pInn_D(Q)$ by the universal property of limits. Moreover, it is easily seen that $\Psi$ is injective. We conclude by showing $\Psi(\Inn(Q)) \subseteq \pInn_D(Q)$ is dense.

Let $x \in \Psi(\Inn(Q))^c$ and let $y \in \Inn(Q_i)$ such that $\pi_i^{-1}(y)$ is a neighborhood of $x.$ We have $\pi_i(x) = y,$ so let $z \in \rho_i^{-1}(y).$ Then $\pi_i(\Psi(z)) = y,$ meaning $\Psi(Z) \in \pi_i^{-1}(y)$, or $\pi_i^{-1}(y) \cap \Psi(\Inn(Q)) \neq \varnothing,$ so $x$ is a limit point of $\Psi(\Inn(Q)).$ Thus, $\Psi(\Inn(Q)) \subseteq \pInn_D(Q)$ is dense.
\end{proof}

\begin{counterexample}\label{nonprofiniteInn}
Proposition \ref{profinitesubgroupisclosed} and Proposition \ref{Inn_dense} now allow us to construct a profinite quandle whose algebraic inner automorphism groups is not profinite.

Let $M_n$ be the subquandle of $\Conj(\mathfrak{S}_n)$ consisting of all $2$-cycles, for $n \geq 2$.  (As an aside, $M_3 = T$).  Observe $\Inn(M_n) \cong \mathfrak{S}_n$ and that symmetries are the transpositions. Now consider $Q = \prod_{n=2}^\infty M_n$. $Q$ is profinite by Theorem \ref{arbitraryproducts}, but for simplicity we consider its profiniteness being witnessed by it being the limit of
\[ M_2 \leftarrow M_2\times M_3 \leftarrow M_2\times M_3\times M_4 \leftarrow \ldots. \]
\noindent  (Here the map from the $n+1$-fold product to the $n$-fold product is projection onto the first $n$ coordinates.)

$\Inn(Q)$ is generated by the symmetries $S_{\langle q_i\rangle_{i=2}^\infty} = \langle S_{q_i}\rangle_{i=2}^\infty$, each of which is of order 2, as each $S_{q_i}$ is of order 2. And the same applies to the finite quandles $\prod_{n=2}^N M_n$ in the directed system.

By Proposition \ref{Inn_dense} it is a dense subgroup of $\pInn_D(Q)$, and thus by Proposition \ref{profinitesubgroupisclosed} would be all of $\pInn(Q)$ were it profinite.

However, it is not all of $\pInn_D(Q)$.  To see this we need to find both $\Inn(Q)$ and $\pInn_D(Q)$.  

Now, $\Inn(\prod_{n=2}^N M_n)$ is the subgroup of $\prod_{n=2}^N \Inn(M_n) = \prod_{n=2}^N \mathfrak{S}_n$ generated by elements which are $N$-tuples of transpositions.  And $\Inn(Q)$ is the subgroup of $\prod_{n=2}^\infty \Inn(M_n) = \prod_{n=2}^\infty \mathfrak{S}_n$ generated by sequences of transpositions, one in each symmetric group. Thus, as elements of the infinite product of the symmetric groups, the coordinates of an element of $\Inn(Q)$ can all be expressed as products of $k$ transpositions, for some fixed $k$. 

However, $\pInn_D(Q)$ is the limit of the sequence of $\Inn(\prod_{n=2}^N M_n)$'s.  Because an element of a symmetric group expressible as a product of $k$ transpositions can also be expressed as a product of $k+2j$ transpositions for $j=0,1,2,...$ this limit contains elements like
\[\ell = \langle \ell_n \rangle_{n=2}^\infty,\]
where 
\[\ell_n = (1 2 \ldots 2\lfloor \frac{n}{2} \rfloor). \]

The projection of this element onto $\Inn(\prod_{n=2}^N M_n)$ has coordinates all of which which can be expressed as a product of $2\lfloor \frac{n}{2} \rfloor$ transpositions.

But $\ell$ is not an element of $\Inn(Q)$, as the number of transpositions need to express its coordinates is unbounded.
\end{counterexample}

While $\pInn_D(Q)$ is dependent on the diagram used witness profiniteness of $Q,$ there are certain properties that are true independent of diagram. This is especially useful because we are afforded advantages from a profinite perspective that we would otherwise lack with the standard inner automorphism group

\begin{proposition}
Let $Q$ be a profinite quandle with inverse system $(\{Q_\lambda\}_{\lambda \in \Lambda},\{\varphi_{\a\b}\}_{\a\leq\b\in\Lambda})$. Then the profinite inner automorphism group $\pInn_D(Q)$ acts continuously on $Q$.
\end{proposition}

\begin{proof}
Define $\rho_\lambda : \Inn(Q_\lambda) \times Q_\lambda \to Q_\lambda$ by $(\varphi, q) \mapsto \varphi(q).$ Each $\rho_\lambda$ is continuous as both spaces are discrete. Now define $\rho: \prod \Inn(Q_\lambda) \times \prod Q_\lambda \to \prod Q_\lambda$ by $\rho = \prod \rho_\lambda.$ Then $\rho$ is continuous by construction, and $\pInn_D(Q) \times Q \subseteq \prod \Inn(Q_\lambda) \times \prod Q_\lambda$ is a subspace, meaning the restriction $\rho|_{\pInn_D(Q) \times Q}$ is also continuous. Furthermore, for $(\varphi, q) = ((\prod_{\lambda \in \Lambda} \varphi_{q'_\lambda}), (\prod_{\lambda \in \Lambda} q_\lambda))$, we have
\[\rho((\varphi, q)) = \prod_{\lambda \in \Lambda} \rho_\lambda(\varphi_{q'_\lambda}, q_\lambda) = \prod_{\lambda \in \Lambda} \varphi_{q'_\lambda}(q_\lambda) \in Q\]
by commutativity of the diagrams $D$ and $\Inn(D)$. Thus $\rho|_{\pInn_D(Q) \times Q}$ is precisely the action by $\pInn_D(Q)$ on $Q$.  
\end{proof}

While $\pInn_D(Q)$ has several nice properties, it is still a construction dependent on choice of diagram. The profinite completion $\widehat{\Inn(Q)}$ of $\Inn(Q)$ admits many of the same properties while arising independent of choice of diagram.

\begin{proposition}
If $Q$ is a profinite quandle, then $\widehat{\Inn(Q)}$ acts continuously on $Q.$
\end{proposition}

\begin{proof}

Let $(\{Q_\lambda\}_{\lambda \in \Lambda},\{\varphi_{\a\b}\}_{\a \leq \b \in \Lambda})$ be an inverse system for $Q$ over a directed poset $\Lambda$ and denote the projections $\pi_\lambda : Q \to Q_\lambda$. For an inner automorphism $S_x \in \Inn(Q)$ where $q \mapsto q \triangleleft x,$ there are induced inner automorphisms $S_{x_\lambda} \in \Inn(Q_\lambda)$ by $q \mapsto q \triangleleft x_\lambda$ where $\pi_\lambda(x) = x_\lambda$ for each $\lambda \in \Lambda,$ so there are induced maps $\rho_\lambda : \Inn(Q) \to \Inn(Q_\lambda)$ where $S_x \mapsto S_{x_\lambda}.$

Therefore, by the universal property of profinite completions, there is a map $\Inn(Q) \to \widehat{\Inn(Q)}$, which induces maps $\widehat{\Inn(Q)} \to \Inn(Q_\lambda)$ for each $\lambda \in \Lambda.$ So by the universal property of limits, there exists a unique continuous homomorphism $\Psi : \widehat{\Inn(Q)} \to \pInn_D(Q)$ - where $D$ is the diagram of the above inverse system - making the diagram below commute:

\[\begin{tikzcd}
	& {\widehat{\Inn(Q)}} \\
	& {\pInn_D(Q)} \\
	{\Inn(Q_\a)} && {\Inn(Q_\b)}
	\arrow["{\rho_\b}"{description}, from=2-2, to=3-3]
	\arrow["{\rho_\a}"{description}, from=2-2, to=3-1]
	\arrow["{\hat\varphi_{ij}}"', from=3-3, to=3-1]
	\arrow["{}", from=1-2, to=3-3]
	\arrow["{}"', from=1-2, to=3-1]
	\arrow["{\exists!\Psi}", dashed, from=1-2, to=2-2]
\end{tikzcd}\]

But then the action $\widehat{\Inn(Q)} \times Q \to Q$ factors through continuous maps

\[\begin{tikzcd}
	{\widehat{\text{Inn}(Q)} \times Q} && {\hspace{5mm}Q\hspace{5mm}} \\
	& {\widehat{\text{Inn}}_D(Q) \times Q}
	\arrow[from=2-2, to=1-3]
	\arrow["{\Psi \times \Id}"', from=1-1, to=2-2]
	\arrow[from=1-1, to=1-3]
\end{tikzcd}\]
Thus, $\widehat{\Inn(Q)}$ acts continuously on $Q$.
\end{proof}

As in the finite setting where unions of orbits of the action of $\Inn(Q)$ give rise of subquandles, in the profinite setting we have

\begin{proposition}
    If $Q$ is a profinite quandle, and $C$ is a compact subset of $Q$, then $Q_C$, the union of the orbits of the elements of $C$ under the action of $\widehat{\Inn(Q)}$, is a profinite subquandle of $Q$.
\end{proposition}

\begin{proof}
    The set $Q_C$ is a subquandle since it is closed under the action of $\widehat{\Inn(Q)}$, and thus {\em a fortiori} under the actions of the symmetries of its own elements and their inverses.  It is also the image of the action map restricted to $C\times\widehat{\Inn(Q)}$.  The continuous image of a compact space is compact and $Q$ is Hausdorff, hence $Q_C$ is closed.  Thus by Proposition \ref{profinitesubisclosed} it is a profinite subquandle.
\end{proof}

Observe that if the hypothesis of compactness were omitted or $\widehat{\Inn(Q)}$ were replaced with $\Inn(Q)$, the union of orbits would still be a subquandle of $Q$ but it would not in general be profinite.

\section{(Profinitely) Algebraically Connected Profinite Quandles}

 We establish a characterization of profinite quandles with a single orbit under action by the profinite completion of their inner automorphism groups analogous to that of Ehrman et al. \cite{ehrman2008toward} in the finite case.

\begin{definition} \label{pacq}
A profinite quandle $Q$ is said to be \textit{(profinitely) algebraically connected} if there is a single orbit under the action by $\widehat{\Inn(Q)}.$
\end{definition}

 As we saw in Counterexample \ref{nonprofiniteInn} above, the inner automorphism group of a profinite quandle is not in general a profinite group.  In the constructions which follow, we will see that actions of profinite analogues of the inner automorphism group give results in the profinite context closely analogous to those involving actions of the inner automorphism group in the finite setting. Therefore in the profinite context, despite the conflict with the usual usage, we will drop the adverb ``profinitely'' and simple call quandles fitting Definition \ref{pacq} ``algebraically connected profinite quandles.''

\begin{definition}
If $G$ is a topological group and $S \subseteq G$, we say that $S$ \textit{topologically generates} $G$ if the topological closure of the subgroup generated by $S$ is equal to $G$.
\end{definition}

The previous definition will be used throughout our discussion of $\widehat{\Inn(Q)}$ as a topological group

For an algebraically connected profinite quandle $Q,$ any choice of $q \in Q$ induces a $\widehat{\Inn(Q)}$-equivariant bijection between $Q$ and $H \backslash \widehat{\Inn(Q)}$ where $H$ is the stabilizer of $q.$ Therefore, we may represent $Q$ by $(\{Hg_\delta\}_{\delta \in \Delta}, \triangleleft)$ where $Hg_1 \triangleleft Hg_2 = Hg_3$ if $q_1 \triangleleft q_2 = q_3.$ We define an augmentation map $|\cdot|: Q = H\backslash \widehat{\Inn(Q)} \to \widehat{\Inn(Q)}$ by $|Hg_\delta| = g$ where $g \in \widehat{\Inn(Q)}$ such that $x \cdot g = x \triangleleft Hg_\delta$ for all $x \in Q$. To distinguish between the augmentation map and the order of a group, we will denote $|H|$ as the order of the subgroup $H$ and $|Hh|$ as the augmentation of $H$ as a right coset in $H \backslash \widehat{\Inn(Q)}.$ 

\begin{theorem}
Let $Q$ be an algebraically connected profinite quandle, let $H$ be the stabilizer of $q \in Q$ under action by $\widehat{\Inn(Q)}$, and let $\{g_\delta\}_{\delta \in \Delta}$ be all coset representatives of $H$ not in $H$. Then $|Hh| \in Z(H)$ and $\widehat{\Inn(Q)}$ is topologically generated by $\{|Hh|, \{|Hg_\delta|\}_{\delta \in \Delta}\}$, where $|Hg_\delta| = g_\delta^{-1}|Hh|g_\delta.$
\end{theorem} 

\begin{proof}
First, since $\{|Hh|, \{|Hg_\delta|\}_{\delta \in \Delta}\}$ attains a representative for each coset of $H,$ the representation from the discussion above implies that $\{|Hh|, \{|Hg_\delta|\}_{\delta \in \Delta}\}$ algebraically generates $\Inn(Q)$.  Since $\Inn(Q)$ is dense in $\widehat{\Inn(Q)}$, we have $\{|Hh|, \{|Hg_\delta|\}_{\delta \in \Delta}\}$ topologically generates $\widehat{\Inn(Q)}$.

Now, by Claim 4.4 of~\cite{ehrman2008toward}, for all $g \in \widehat{\Inn(Q)}$ we have $|Hg_\delta g| = g^{-1}|Hg_\delta|g$. Then  $|Hg_\delta| = |Hhg_\delta| = g_\delta^{-1}|Hh|g_\delta.$ For $g \in H$, we have $|Hh| = |Hhg| = g^{-1}|Hh|g$, meaning $|Hh| \in Z(H)$.
\end{proof}

\begin{proposition}
For a group $G$, a proper subgroup $H < G$, and some $h \in Z(H),$ define $\triangleleft : H \backslash G \times H \backslash G \to H \backslash G$ by $(Hg, Hk) \mapsto Hgk^{-1}hk$ and $\triangleleft^{-1} : H \backslash G \times H \backslash G \to H \backslash G$ by $(Hg, Hk) \mapsto Hgk^{-1}h^{-1}k$. Then $Q_{G,H,h} := (H \backslash G, \triangleleft, \triangleleft^{-1})$ is a quandle, called the \textit{right coset quandle of G with respect to H}.
\end{proposition}

\begin{proof}
Let $Hg,Hk,Hm \in Q_{G,H,h}$ be distinct.

\begin{enumerate}[label=\textbf{(Q\arabic*)}]
\item We have 
    \[Hg \triangleleft Hg = Hgg^{-1}hg = Hhg = Hg.\]
\item We have
    \begin{align*}
        (Hg \triangleleft Hm) \triangleleft (Hk \triangleleft Hm) 
        &= Hgm^{-1}hm \triangleleft Hkm^{-1}hm \\
        &= Hgm^{-1}hmm^{-1}h^{-1}mk^{-1}hkm^{-1}hm \\
        &= Hgk^{-1}hkm^{-1}hm \\
        &= Hgk^{-1}hk \triangleleft Hm \\
        &= (Hg \triangleleft Hk) \triangleleft Hm.
    \end{align*}
\item We have
    \begin{align*}
        (Hg \triangleleft Hk) \triangleleft^{-1} Hk &= Hgk^{-1}hk \triangleleft^{-1} Hk = Hgk^{-1}hkk^{-1}h^{-1}k = Hg, \\
        (Hg \triangleleft^{-1} Hk) \triangleleft Hk &= Hgk^{-1}h^{-1}k \triangleleft Hk = Hgk^{-1}h^{-1}kk^{-1}hk = Hg. \\
    \end{align*}
\end{enumerate}
Thus, $Q_{G,H,h}$ is a quandle.
\end{proof}

Given group homomorphisms, there is an induced quandle homomorphism on their respective right coset quandles.

\begin{lemma}
    For a group $G$, let $H < G$ and $h \in Z(H)$. Define a surjective group homomorphism $\varphi: G \to \Gamma.$ Then there is an induced surjective quandle homomorphism $\hat\varphi: Q_{G,H,h} \to Q_{\Gamma, \varphi(H),\varphi(h)}$ by $Hg \mapsto \varphi(H)\varphi(g)$.
\end{lemma}

\begin{proof}
First, note $\varphi(h) \in Z(\varphi(H)).$ Let $Hg,Hk \in Q_{G,H,h}.$ Then 
\begin{align*}
    \hat\varphi(Hg \triangleleft_G Hk) &= \hat\varphi(Hgk^{-1}hk)  \\
    &= \varphi(H)\varphi(g)\varphi(k^{-1})\varphi(h)\varphi(k) \\
    &= \varphi(H)\varphi(g)\varphi(k)^{-1}\varphi(h)\varphi(k) \\
    &= \varphi(H)\varphi(g) \triangleleft_\Gamma \varphi(H)\varphi(k), \\
\end{align*}
hence $\hat\varphi$ is a quandle homomorphism. 

To see that $\hat\varphi$ is surjective, let $\varphi(H)\gamma \in Q_{\Gamma,\varphi(H),\varphi(h)}.$ Since $\varphi$ is surjective, there exists some $g \in G$ such that $\varphi(g) = \gamma.$ Then $\hat\varphi(Hg) = \varphi(H)\varphi(g) = \varphi(H)\gamma.$

\end{proof}

\begin{theorem}
Let $G$ be a profinite group and $H < G$. Further, for coset representatives $\{g_\delta\}_{\delta \in \Delta}$ of $H$, suppose that $G$ is topologically generated by $\{g_\delta^{-1}hg_\delta\}_{\delta \in \Delta}$ for $ h \in Z(H).$ Then $Q_{G,H,h}$ is an algebraically connected profinite quandle.
\end{theorem}

\begin{proof}
Note coset representative notation is omitted when useful.

Let $(\{G_\lambda\}_{\lambda \in \Lambda},\{\varphi_{\a\b}\}_{\a \leq \b \in \Lambda})$ be an inverse system for $G$ over a directed poset $\Lambda$ and denote the projections $\pi_\lambda : G \to G_\lambda$. We begin by showing that $Q_{G,H,h}$ is profinite. From the above lemma, there are induced surjective maps $\hat\pi_\lambda : Q_{G,H,h} \to Q_{G_\lambda, \pi_\lambda(H),\pi_\lambda(h)}$ and $\hat\varphi_{\a\b}: Q_{G_\b, \pi_\b(H),\pi_\b(h)} \to Q_{G_\a, \pi_\a(H),\pi_\a(h)},$ meaning $(\{Q_{G_\lambda, \pi_\lambda(H),\pi_\lambda(h)}\}_{\lambda \in \Lambda},\{\hat\varphi_{\a\b}\}_{\a\leq\b \in \Lambda})$ forms an inverse system. By the universal property of limits, we obtain a unique continuous quandle homomorphism $\Psi: Q_{G,H,h} \to \varprojlim Q_{G_\lambda, \pi_\lambda(H),\pi_\lambda(h)}$ such that the following diagram commutes.

\[\begin{tikzcd}
	& {Q_{G, H,h}} \\
	& {\varprojlim Q_{G_\lambda, \pi_\lambda(H),\pi_\lambda(h)}} \\
	{Q_{G_\a, \pi_\a(H),\pi_\a(h)}} && {Q_{G_\b, \pi_\b(H),\pi_\b(h)}}
	\arrow["{\rho_\b}"{description}, from=2-2, to=3-3]
	\arrow["{\rho_\a}"{description}, from=2-2, to=3-1]
	\arrow["{\hat\varphi_{\a\b}}"', from=3-3, to=3-1]
	\arrow["{\hat\pi_\b}", from=1-2, to=3-3]
	\arrow["{\hat\pi_\a}"', from=1-2, to=3-1]
	\arrow["{\exists!\Psi}", dashed, from=1-2, to=2-2]
\end{tikzcd}\]

We claim $\Psi$ is an isomorphism. To see that $\Psi$ is injective, let $q_1,q_2 \in Q_{G,H,h}$ such that $\Psi(q_1) = \Psi(q_2).$ Then for all $\lambda \in \Lambda,$
\[\rho_\lambda(\Psi(q_1)) = \rho_\lambda(\Psi(q_2)) \Rightarrow \hat\pi_\lambda(q_1) = \hat\pi_\lambda(q_2) \Rightarrow q_1 = q_2.\]

For surjectivity, we have $Q_{G,H,h}$ is a (topological) quotient of $G$ which is a Stone space, hence $Q_{G,H,h}$ is compact. Then $\Psi(Q_{G,H,h}) \subseteq \varprojlim Q_{G_\lambda, \pi_\lambda(H),\pi_\lambda(h)}$ is also compact.  We have by the diagram that $\rho_\alpha(\Psi(Q_{G,H,h})) = \hat\pi_\alpha(Q_{G,H,h}) = Q_{G_\alpha, \pi_\alpha(H),\pi_\alpha(h)}$ by the surjectivity of $\hat\pi_\alpha$. Hence by \ref{density} the subquandle $\Psi(Q_{G,H,h})$ is dense in $\varprojlim Q_{G_\lambda, \pi_\lambda(H),\pi_\lambda(h)}$ and 
\[Q_{G,H,h} \cong \varprojlim Q_{G_\lambda, \pi_\lambda(H),\pi_\lambda(h)}.\]


Finally, we show that $Q_{G,H,h}$ is algebraically connected. Since $G$ is topologically generated by $\{g_\delta^{-1}hg_\delta\}_{\delta \in \Delta}$, the action on $Q_{G,H,h}$ by $\widehat{\Inn(Q_{G,H,h})}$ is transitive (the action is just the multiplication action on cosets of a subgroup). Thus there is a single orbit under this action and $Q_{G,H,h}$ is algebraically connected.
\end{proof}

The previous two theorems provide a program for constructing all algebraically connected profinite quandles.

\section{Directions for future research}

The results herein leave our knowledge of profinite quandles in a state slightly less satisfactory that that attained in \cite{ehrman2008toward} Ehrman et al. in the finite case.  There the notion of semi-disjoint union gave a way of constructing all finite quandles from algebraically connected finite quandle by iterated semi-disjoint union.  We lack an analogous constuction.  The example of the ``Cantor quandle'', that is the trivial quandle structure on the Cantor set with the usual topology inherited from ${\mathbb R}$ shows that topological considerations must be taken into account:  algebraically the Cantor quandle is an uncountable disjoint union of singleton trivial quandles (and thus a semi-disjoint union with trivial actions), but the topology which makes it profinite cannot be accounted for by the semi-disjoint union construction of \cite{ehrman2008toward} Ehrman et al.  One line of future research would be to investigate the sufficiency of a construction akin to the semi-disjoint union construction, performed not on a set of profinite quandles, but a family of profinite quandles indexed by a Stone space, together with transfinite induction as a means of constructing all profinite quandles from algebraically connected profinite quandles.

A second obvious line of research is the resolution the conjecture of \cite{amsberry2023complementation} Amesbury, et al. which originally motivated the project:  is the lattice of subquandles of a profinite quandle always complemented?  We will remark here that the lattice of closed subquandles of a profinite quandle is not in general complemented.  The Cantor quandle mentioned above provides the needed counterexample. Observe that any subquandle in the Cantor quandle consisting of a single element is closed but does not have a closed complement.  The singleton subquandles do have a complement in the full subquandle lattice -- their set-theoretic complement -- since the subquandle lattice of a trivial quandle is simply the Boolean algebra of its subsets.  A subsidiary question to the conjecture of Amesbury et al. is whether close subquandles of a profinite quandle have complements in the full subquandle lattice.

Finally, the authors hope that the results herein together the interpretation of quandles as a means of encoding monodromy \cite{yetter2003} will prove useful in the applications of profinite quandles to the \'{e}tale homotopy theory of number fields found in forthcoming work of Davis and Schlank \cite{Schlank2023}.

\section{Acknowledgements}

This paper is the result of a program of research begun under the mentorship of the fifth author while the first four authors were participants at Kansas State University's mathematics REU site, SUMaR 2023.  The first four authors wish to thank Kansas State for hospitality, and all of the authors wish to thank the National Science Foundation for its financial support of this project under award \#DMS-2243854.


\bibliography{\jobname}

\end{document}